\documentclass[11pt]{article}

\usepackage{float}
\usepackage{multirow}
\usepackage{amsmath,amsfonts,amssymb,graphicx,amsthm,color,natbib}
\usepackage{authblk}
\allowdisplaybreaks
\usepackage[]{geometry}
\usepackage[normalem]{ulem}
\usepackage{hyperref}

\newtheorem{proposition}{Proposition}%
\newtheorem{corollary}{Corollary}

\begin{document}
	
	\title{A family of consistent normally distributed tests for Poissonity}

\author[1]{{Antonio} {Di Noia}}
\author[1]{{Marzia} {Marcheselli}}
\author[1]{{Caterina} {Pisani}}
\author[2]{{Luca} {Pratelli}}
\affil[1]{Department of Economics and Statistics, University of Siena}
\affil[2]{Italian Naval Academy}
\date{}
\setcounter{Maxaffil}{0}
\renewcommand\Affilfont{\itshape\small}
\maketitle
\let\thefootnote\relax\footnotetext{\emph{Email addresses:} antonio.dinoia55@gmail.com (Antonio Di Noia), marzia.marcheselli(\sout{at})unisi.it (Marzia Marcheselli), caterina.pisani(\sout{at})unisi.it (Caterina Pisani), luca\_pratelli(\sout{at})marina.difesa.it (Luca Pratelli).}

\begin{abstract}
A {family of} consistent tests, {derived from} a characterization of the probability generating function, is proposed for assessing Poissonity against a wide class of count distributions, which includes some of the most frequently adopted alternatives to the Poisson distribution. {Actually, the family of test statistics is based on the difference between the plug-in estimator of the Poisson cumulative distribution function and the empirical cumulative distribution function. The test statistics have an intuitive and simple form and are} asymptotically normally distributed, allowing a straightforward implementation of the test. The finite sample properties of the test are investigated by means of an extensive simulation study. The test shows satisfactory behaviour compared to other tests with known limit distribution.
\end{abstract}

\noindent {\bf Keywords:}
asymptotic normality, consistent test, Poissonity, probability generating function.

\maketitle

\section{Introduction}
Assessing the Poissonity assumption is a relevant issue of statistical inference, both because Poisson distribution has an impressive list of applications in biology, epidemiology, physics, and queue theory (see e.g. \citealp{Johnson}; \citealp{PuigWei}) and because it is a preliminary step in order to apply many popular statistical models. 
The use of the probability generating function (p.g.f.) has a long tradition (see e.g. \citealp{Kocherlakota}; \citealp{MeintanisBassiakos}; \citealp{Remillard}) for testing discrete distributions and some omnibus procedures based on the p.g.f. have been proposed for Poissonity (e.g. \citealp{Nakamura}; \citealp{Baringhaus}; \citealp{Rueda}; \citealp{Gurtler}; \citealp{Meintanis}; \citealp{Inglot}; \citealp{PuigWei}). 
Omnibus tests are particularly appealing since they are consistent against all possible alternative distributions but they commonly have a non-trivial asymptotic behavior. Moreover, the distribution of the test statistic may depend on the unknown value of the Poisson parameter, implying the necessity to use computationally intensive bootstrap, jackknife, or other resampling methods to approximate it. 
On the other hand, Poissonity tests against specific alternatives may achieve high power but rely on the knowledge of what deviations from Poissonity can occur.
An alternative approach, proposed by \cite{Meintanis}, is to consider tests with suitable asymptotic properties with respect to a fairly wide class of alternatives, which are also the most likely when dealing with the Poissonity assumption.

In this paper, by referring to the same class of alternative distributions and by using the characterization of the Poisson distribution based on its p.g.f., we propose a family of consistent and asymptotically normally distributed test statistics{, based on the difference between the plug-in estimator of the Poisson cumulative distribution function (c.d.f.) and the empirical c.d.f.,}  and a data-driven procedure for the choice of the parameter indexing the statistics. In particular, the test statistics not only have an intuitive interpretation but, being simple to compute, allow a straightforward implementation of the test and lead to test procedures with satisfactory performance also in presence of contiguous alternatives.

\section{Characterization of the Poisson distribution}
Let $X$ be a random variable (r.v.) taking natural values with probability mass function $p_X$ and ${\rm E}[X]=\mu$. Moreover, let $\Psi_{X}(t)={\rm E}[t^X]$, with $t\in[0,1]$, be the p.g.f. of $X$. Following \cite{Meintanis}, we consider the class of count distributions $\Delta$ such that 
\begin{equation*}
	\label{diffeq}
	D(t,\mu)=\Psi_{X}^\prime(t)-\mu\Psi_{X}(t)
\end{equation*} 
is not negative for any $t\in[0,1]$ or not positive for any $t\in[0,1]$ for all $\mu>0$, where $\Psi_{X}^\prime(t)$ is the first order derivative of the p.g.f.. 

{As proven by \cite{Meintanis},} this class contains many popular alternatives to the Poisson distribution, such as the Binomial distribution, the Negative Binomial distribution, the generalized Hermite distribution, the Zero-Inflated {and generalized Poisson distribution, among others}.

It must be pointed out that $D(t,\mu)=0$ for any $t\in[0,1]$ and for some $\mu>0$ if and only if $X$ is a Poisson r.v.. This characterization allows to construct a goodness of fit test for Poissonity against alternatives belonging to the class $\Delta$, that is for the hypothesis system 
\begin{equation*}
	\begin{aligned}
		H_0&: X \in \Pi_\mu, \text{for some}\ \mu>0 \cr
		H_1&: X \in \Delta 
	\end{aligned}
\end{equation*}
where $\Pi_\mu$ denotes the Poisson distribution with parameter $\mu$. 
In particular, \cite{Meintanis} adopt the previous characterization to construct a consistent test for the Poisson distribution by means of the empirical counterpart of $D(t,\mu)$ suitably weighted.
An alternative approach can be based on the $L^1$ distance of $D(t,\mu)$ from $0$, whose positive values evidence departures from Poissonity. The following Proposition, giving bounds for this distance, provides insight into the introduction of a family of test statistics. 

\begin{proposition} For any natural number $k$, let $f_k(\mu)=e^{-\mu}(1+\ldots+{\frac {\mu^k}{ k!}}).$ For any $X \in \Delta$ and for any $\mu>0$, it holds  
	\begin{equation}\label{Eq 1} 
		\begin{aligned}
			\frac{1}{1+\ldots+{\frac {\mu^k}{ k!}}}\vert  T^{(k)}\vert \leq\int_0^1 \vert D(t,\mu)\vert\, dt \leq &{e^{\mu}}\vert T^{(k)}\vert,
		\end{aligned}
	\end{equation}
	where $T^{(k)}=f_k(\mu)-p_X(0)(1+\ldots+{\frac{\mu^k}{ k!}}).$ 	
	In particular, for $k=0$,
	\begin{equation*}\label{Eq 2} 
		\begin{aligned}
			\vert e^{-\mu}-p_X(0)\vert\leq\int_0^1 \vert D(t,\mu)\vert\, dt \leq &e^\mu\,\vert e^{-\mu}-p_X(0)\vert.
		\end{aligned}
	\end{equation*}
\end{proposition}
\begin{proof}
	Since $X \in \Delta$ 
	\begin{equation*}
		\begin{aligned}
			\big\vert\int_0^1 D(t,\mu)e^{-\mu t}\, dt\big\vert &\leq \int_0^1 \vert D(t,\mu)\vert\, dt 
			=\int_0^1 e^{\mu t}\vert D(t,\mu)e^{-\mu t}\vert \, dt\\ &\leq e^\mu\int_0^1 \vert D(t,\mu)e^{-\mu t}\vert \, dt=
			e^\mu\big\vert\int_0^1 D(t,\mu)e^{-\mu t}\, dt\big\vert.
		\end{aligned}
	\end{equation*}
	As $D(t,\mu)e^{-\mu t}=\{\Psi_{X}(t)e^{-\mu t}\}^\prime$, then
	\begin{equation*}
		\begin{aligned}
			\int_0^1 D(t,\mu)e^{-\mu t}\, dt = e^{-\mu }-p_X(0).
		\end{aligned}
	\end{equation*}
	By dividing and multiplying for $1+\ldots+{\frac{\mu^k}{ k!}}$, the thesis immediately follows.
\end{proof}

\medskip\noindent Thanks to inequality \eqref{Eq 1}, a family of test statistics, indexed by $k$ and depending on an estimator of 
$$T^{ (k)}=f_k(\mu)-p_X(0)\Big(1+\ldots+{{\mu^k}\over{ k!}}\Big)$$ can be defined. Note that $f_k(\mu)$ and $p_X(0)(1+\ldots+{\frac{\mu^k}{ k!}})$ are equal to $P(X\leq k)$ when $X$ belongs to $\Pi_\mu$ and obviously  
$T^{(k)}\neq 0$ for some $k\geq 0$ iff $X$ is not a Poisson r.v. iff $T^{(0)}\neq 0$.

\section{The test statistic}

Given a random sample $X_1,\ldots,X_n$ from $X$, let $\overline X_n$ 
be the sample mean {and 
	$$\widehat p_X(0)=\frac{I{(X_1= 0)}+\ldots+I{(X_n= 0)}}{n}.$$}
\noindent The simplest test statistic arises from the estimator  {$$\widehat T^{ (0)}=e^{-\overline X_n }-\widehat p_X(0)$$ of $T^{ (0)}$. This statistic is really appealing also owing to its straightforward interpretation, being based on the comparison of}  the probability that $X$ takes value zero with the probability of zero for a Poisson r.v. {Unfortunately,} its performance may  {be not satisfactory,} especially when the sample size is small while $\mu$ is relatively large, as the estimation of $p_X(0)$ becomes even more crucial. 

{Nevertheless, for $k>0$, as $T^{ (k)}=T^{ (0)}(1+\ldots+{{\mu^k}\over{ k!}})$, the natural estimator of $T^{ (k)}$, given by  $\widehat T^{ (0)}(1+\ldots+{{{\overline X_n} ^k}\over{ k!}})$, suffers from the same drawbacks of $\widehat T^{ (0)}$. To avoid the estimation of $p_X(0)$, since $p_X(0)(1+\ldots+{\frac{\mu^k}{ k!}})$ is equal to $P(X\leq k)$ under $H_0$, the following estimator of $T^{(k)}$ is proposed}
$$\widehat T_n^{ (k)}=f_k(\overline X_n)-F_n(k),$$
where 
\begin{equation*}
	F_n(k)=\frac{I{(X_1\leq k)}+\ldots+I{(X_n\leq k)}}{n}.
\end{equation*}

{It is worth noting that, since $k$ is a fixed natural number (often $k=0$), $\sqrt{n}\widehat T_n^{(k)}$ is the $k$-th r.v. of the estimated (discrete) empirical process introduced by \cite{henze1996empirical} for dealing with goodness-of-fit tests
	for discrete distributions and also considered by \cite{Gurtler} in their critical synopsis of several procedures for assessing Poissonity.
	However, as $\sqrt{n}\widehat T_n^{ (k)}$ is a r.v., its}  asymptotic distribution can be easily derived from classical Central Limit Theorems {under very mild assumptions}, as shown in the following Proposition.

\begin{proposition}
	\label{prop:conv}
	Let $X$ be a r.v. with
	${\rm Var}[X]$ finite and $k$ be a fixed natural number.  Let 
	\begin{equation*}
		V^{ (k)}_n=\sqrt{n}\, \widehat T_n^{ (k)}.
	\end{equation*}
	If $X\in \Pi_\mu$ then $V_n^{ (k)}$ converges in distribution to ${\mathcal N}(0,\sigma^2_{\mu,k})$ as $n\to \infty$, where 
	\begin{equation}\label{var}
		\sigma^2_{\mu,k}={\rm Var}\Big[e^{-\mu}{\frac{\mu^{k}}{ k!}}X+I(X\leq k)\Big]=f_k(\mu)\{1-f_k(\mu)\}-{\frac{e^{-2\mu}\mu^{2k+1}}{(k!)^2}}.
	\end{equation}
	Moreover, if $r_k=P(X\leq k)-f_k(\mu)\neq 0$ for some $k,$ namely $X$ is not a Poisson r.v., then $\vert V_n^{ (k)}\vert $ converges in probability to $\infty$.
\end{proposition}
\begin{proof}
	Note that $f^\prime_k(\mu)=-e^{-\mu}\frac{\mu^k}{k!}.$ Owing to the Delta Method 
	\begin{equation*}
		\begin{aligned}
			V_n^{ (k)} &=\sqrt n\{f_k(\overline X_n)-f_k(\mu)+f_k(\mu)-F_n(k)\}\cr 
			&=-\sqrt n\{e^{-\mu}{\frac{\mu^{k}}{ k!}}(\overline X_n-\mu)+F_n(k)-f_k(\mu)\}+o_P(1)\\
			&={\frac{g(X_1)+\ldots+g(X_n)}{\sqrt n}}-\sqrt n r_k+o_P(1),
		\end{aligned}
	\end{equation*}
	where $g$ is the function defined by $$x\mapsto -e^{-\mu}{\frac{\mu^{k}}{ k!}}(x-\mu)-\{I(x\leq k)-P(X\leq k)\}.$$
	When $X\in\Pi_\mu$, $r_k=0$, ${\rm E}[g(X)]=0$ and ${\rm E}[o_P^2(1)]$ is $o(1)$ since we have $$\vert f_k(\overline X_n)-f_k(\mu)\vert\leq \vert\overline X_n-\mu\vert.$$ Then, under $H_0$, {$V_n^{ (k)}$ converges in distribution to ${\mathcal N}(0,{ {\rm Var}\, [g(X)]})$ owing to the Central Limit Theorem.}  Moreover, since
	\begin{equation*}
		{\rm Var}\, [g(X)]={\frac{e^{-2\mu}\mu^{2k+1}}{(k!)^2}}+f_k(\mu)\{1-f_k(\mu)\}+2e^{-\mu}{\frac{\mu^{k}}{ k!}}{\rm Cov}\, \big[X,I(X\leq k)\big]
	\end{equation*}
	and
	$${\rm Cov}\, \big[X,I(X\leq k)\big]=\sum_{j=1}^k \frac{e^{-\mu}\mu^{j}}{(j-1)!}-\mu f_k(\mu)=-e^{-\mu}{\frac{\mu^{k+1}}{ k!}}, $$ {relation \eqref{var}, and consequently the first part of the proposition, is proven.}
	
	Now, let $X$ be a r.v. such that $r_k\neq 0$, in particular $X$ is not a Poisson r.v.. Thus, $\vert V_n^{ (k)}\vert $ converges in probability to $\infty$ because $${\frac{g(X_1)+\ldots+g(X_n)}{\sqrt n}}+o_P(1)$$ is bounded in probability and $\sqrt n \vert r_k\vert $ converges to $\infty$. The second part of the proposition is so proven.	
	
\end{proof}
\medskip\noindent Thanks to Proposition \ref{prop:conv}, fixed a natural number $k\geq 0$ and under the null hypothesis $H_0$, $\sqrt n\, \widehat T_n^{ (k)}/\sigma_{\mu,k}$ converges to ${\mathcal N}(0,1)$. Therefore, an estimator of $\sigma_{\mu,k}$ is needed to define the test statistic.
As the plug-in estimator
$$\widetilde\sigma^2_{n,k}=e^{-2\overline X_n}\big\{\big(1+\ldots+{{\overline X_n^k}\over{ k!}}\big)\big(e^{\overline X_n}-1-\ldots-{{\overline X_n^k}\over{ k!}}\big)-{{\overline X_n^{2k+1}}\over{ (k!)^2}}\big\}$$
converges a.s. and in quadratic mean to $\sigma^2_{\mu,k}$, for any natural number $k$, the test statistic turns out to be 
\begin{equation*}
	Z_{n,k}=\frac{\sqrt n\, \widehat T_n^{ (k)}}{\widetilde\sigma_{n,k}}.
\end{equation*}
An $\alpha$-level large sample test rejects $H_0$ for realizations of the test statistic whose absolute values are greater than $z_{ 1-\alpha/2}$, where $z_{ 1-\alpha/2}$ denotes the ${ 1-\alpha/2}$-quantile of the standard normal distribution.

\begin{corollary} 
	\label{cor:H_1}
	Under $H_1$, for any natural number $k$ such that $P(X\leq k)-f_k(\mu)\neq 0$, $Z_{n,k}$ converges in probability to $\infty$. 
\end{corollary}
\begin{proof}
	As $\widetilde\sigma^2_{n,k}$ converges a.s. to $\sigma^2_{\mu,k}$, the proof immediately follows from the second part of Proposition \ref{prop:conv}.
\end{proof}
\medskip\noindent It is at once apparent that $Z_{n,k}$ actually constitutes a family of test statistics giving rise to consistent test for $k=0$ and for all the other values of $k$ for which there is  a discrepancy between the cumulative distribution of the Poisson and of $X$. {Among this family of test statistics, only $Z_{n,0}$  belongs to the family of the Poisson zero indexes \citep{weiss2019testing}. It is particularly attractive owing to its simplicity} but, {as already pointed out,} its finite-sample performance may deteriorate, especially if the sample size is small and $\mu$ is relatively large. Therefore, the selection of the parameter $k$ ensuring consistency and good discriminatory capability is crucial and a data-driven selection criterion is proposed.

\section{Data-driven choice of $k$}
An heuristic, relatively simple, criterion for choosing $k$ is based on the relative discrepancy measure $T_n^{ (k)}/f_k(\mu)$ which can be estimated by $\frac{\widetilde\sigma_{n,k}}{f_k(\overline X_n)\sqrt n}Z_{n,k}$. Recalling that, under $H_0$, $Z_{n,k}$ is approximately a standard normal r.v. also for moderate sample size, as 
$\frac{\widetilde\sigma_{n,k}}{f_k(\overline X_n)\sqrt n}$ converges a.s. to $0$ when $n\to\infty$, for any fixed $n$, $k$ may be selected in a such a way that  $\frac{\widetilde\sigma_{n,k}}{f_k(\overline X_n)\sqrt n}$ is not negligible. This choice should ensure both high power and an actual significance level close to the nominal one. To this purpose, note that the function $\mu\mapsto f_k(\mu)$ is decreasing for any $k$ and if $\mu\geq1$, it holds $f^2_k(\mu)/f_k(1)\leq f_k(\mu)$ and  $\sup_k 1/f_k(\mu)=1/f_0(1)=e$. Then $k$ can be selected as the smallest natural number such that $\frac{\widetilde\sigma_{n,k}}{f^2_k(\overline X_n)\sqrt n}$ is not greater than $e$ when $\overline X_n\geq 1$, that is 
$$k^*_n=\min\Big\{k\geq 0: I(\overline X_n\geq 1)\,\frac{\widetilde\sigma_{n,k}}{f^2_k(\overline X_n)\sqrt n}\leq e\Big\}. $$
Notwithstanding $k^*_n$ converges a.s. to $0$ since ${\widetilde\sigma_{n,k}}\leq 1/2$ from \eqref{var}, the convergence rate may be very slow for large values of $\mu$ in such a way that $k^*_n$ can be rather larger than $0$ even for large sample sizes. Finally, by considering the  test statistic corresponding to $k^*_n$
\begin{equation*}
	W_n=\frac{\sqrt n\, \widehat T_n^{ (k^*_n)}}{\widetilde\sigma_{n,k^*_n}},
\end{equation*}
its asymptotic behaviour can be obtained. {Obviously in this case $\sqrt n\, \widehat T_n^{ (k^*_n)}$ is no more a r.v. belonging to the estimated empirical process \citep{henze1996empirical}.}
\begin{corollary} 
	Under $H_0$, $W_n$ converges in distribution to ${\mathcal N}(0,1)$ as $n\to \infty$ and, under $H_1$, $W_n$ converges in probability to $\infty$. 
\end{corollary}
\begin{proof}
	Since $k^*_n$ converges to 0, $W_n$ and $Z_{n,0}$ have the same asymptotic behaviour and the proof immediately follows from Proposition \ref{prop:conv} and Corollary \ref{cor:H_1}.  
\end{proof}
It is worth noting that the selection of $k_n^*$ by means of the proposed data-driven criterion ensures consistency, maintaining the asymptotic normal distribution of the test statistic.

\section{Asymptotic behaviour under contiguous alternatives}

The asymptotic behaviour of the proposed test is investigated for detecting Poisson departures from {contiguous alternatives. More precisely,} given a positive number $\lambda$, for any $n\geq\lambda^2$, let {$ X^{(n)}$ be a mixture of r.v.s given by}
\begin{equation}
	\label{mix}
	X^{(n)}\sim I(A_n)X+I(A^c_n)Y
\end{equation}
where $X\in\Pi_\mu$, $Y$ is a positive random variable with ${\rm E}[Y]=\mu$ and ${\rm E}[Y^2]<\infty$,  $X,Y,I(A_n)$ are  independent and $P(A_n)=(1-\frac{\lambda}{\sqrt n})$. 
{Roughly speaking, $\lambda$ represents a parameter quantifying the discrepancy between the distribution of $X$ and the distribution of $X^{(n)}$. Obviously, for small $\lambda$ values detecting departures from Poisson is extremely difficult.}
Note that $X^{(n)}$ belongs to $\Delta$ if $Y$ belongs to $\Delta$ and converges to a Poisson r.v.. 
\begin{proposition} For any $n\geq\lambda^2$, given a random sample $X^{(n)}_1,\ldots, X^{(n)}_n$ from $X^{(n)}$, for a fixed natural number $k$, let
	\begin{equation*}
		U^{ (k)}_n=\sqrt{n}\,\big\{f_k\big(\overline X_n^{\prime}\big)-F^{\prime}_n(k)\big\}
	\end{equation*}
	where $\overline X_n^{\prime}$ and $F^{\prime}_n$ are the sample mean and the empirical cumulative distribution function. Then $U_n^{ (k)}$ converges in distribution to ${\mathcal N}(\tau_k,\sigma^2_{\mu,k})$ as $n\to \infty$, where 
	$$\tau_k=-\lambda \{P(Y\leq k)-f_k(\mu)\}.$$ In particular $U^{(k^*_n)}_n$, where $k^*_n$ is obtained by the data-driven criterion, is equivalent to  $U^{ (0)}_n$, which converges in distribution to ${\mathcal N}(\lambda\{e^{-\mu}-P(Y=0)\},e^{-2\mu}(e^{\mu}-1-\mu))$.
\end{proposition}
\begin{proof}
	Note that $\vert f^{\prime\prime}_k\vert \leq 2$ for any $k$.  Owing to the Taylor's Theorem
	
	$$\vert \sqrt{n}\,\{f_k\big(\overline X_n^{\prime}\big)-f_k(\mu)\}+\sqrt ne^{-\mu}{\frac{\mu^{k}}{ k!}}\big(\overline X_n^{\prime}-\mu)\vert \leq \sqrt{n} (\overline X_n^{\prime}-\mu)^2$$
	Since $${\rm E}\big[\sqrt{n}(\overline X_n^{\prime}-\mu)^2\big]\leq \frac{\mu+{\rm Var}[Y]}{\sqrt n},$$
	it follows that $U^{ (k)}_n$ has the same asymptotic behaviour of $${\frac{g_n(X^{(n)}_1)+\ldots+g_n(X^{(n)}_n)}{\sqrt n}}-\sqrt n r_{n,k},$$ where $g_n$ is the function defined by $$x\mapsto -e^{-\mu}{\frac{\mu^{k}}{ k!}}(x-\mu)-\{I(x\leq k)-P(X^{(n)}\leq k)\}$$ and $r_{n,k}=P(X^{(n)}\leq k)-f_k(\mu)$.
	Since $\lim_n -\sqrt n r_{n,k}=\tau_k$, the thesis follows from the convergence in distribution of  ${\frac{g_n(X^{(n)}_1)+\ldots+g_n(X^{(n)}_n)}{\sqrt n}}$ to ${\mathcal N}(0,\sigma^2_{\mu,k}).$ 
\end{proof}
\medskip\noindent  The previous proposition can be considered a non-parametric version of classical asymptotic analysis under the so-called shrinking alternative. Moreover, the test statistic has a local asymptotic normal distribution which is useful to highlight its discriminatory capability under not trivial contiguous alternatives. In a parametric setting, by means of Le Cam lemmas \citep{Lecam}, it could be possible to derive the limiting power function and to build an efficiency measure for test statistics. Clearly, in a non-parametric functional setting, a closed form of the power function is not available and must be assessed by means of simulation studies.

\section{Simulation study}
The performance of the proposed test has been assessed by means of an extensive Monte Carlo simulation. First of all, fixed the nominal level $\alpha=0.05$, the significance level of the test is empirically evaluated, as the proportion of rejections of the null hypothesis, by independently generating $10000$ samples of size $n=50$ from Poisson distributions with $\mu$ varying from $1$ to $16$ by ${0.5}$. As early mentioned, the family of test statistics $Z_{n,k}$ depends on the parameter $k$ and therefore, for any $\mu$, the empirical significance level is computed for $k=0,1,2,3$ and reported in Figure \ref{fig:figure1}. Simulation results confirm that for large values of $\mu$ the empirical level is far from the nominal one even for a reasonably large sample size and that a data-driven procedure is needed to select $k$. Thus, the test statistic $W_n$ is considered, and its performance is compared to those of two tests having known asymptotic distributions: the test by \cite{Meintanis}, $MN_n$, also recommended by \cite{mijburgh2020overview} to achieve good power against a large variety of deviations from the Poisson distribution, and the Fisher index of dispersion, $ID_n$, which, owing to its simplicity, is often considered as a benchmark.
The explicit ready-to-implement test statistic $MN_n$ has a non-trivial expression and it is based on $\int_0^1D(t,\mu)t^adt,$ where $a$ is a suitable parameter. $MN_n$ is proven to have an asymptotic normal distribution. In the simulation, $a$ is set equal to $3$ as suggested when there is no prior information on the alternative model.
The Fisher index of dispersion test is performed as an asymptotic two-sided chi-square test and it is based on the extremely simple test statistic $ID_n={\sum_{i=1}^n(X_i-\overline{X}_n)^2}/{\overline{X}_n}.$

\begin{figure} 
	\centering
	\includegraphics[width=0.65\textwidth]{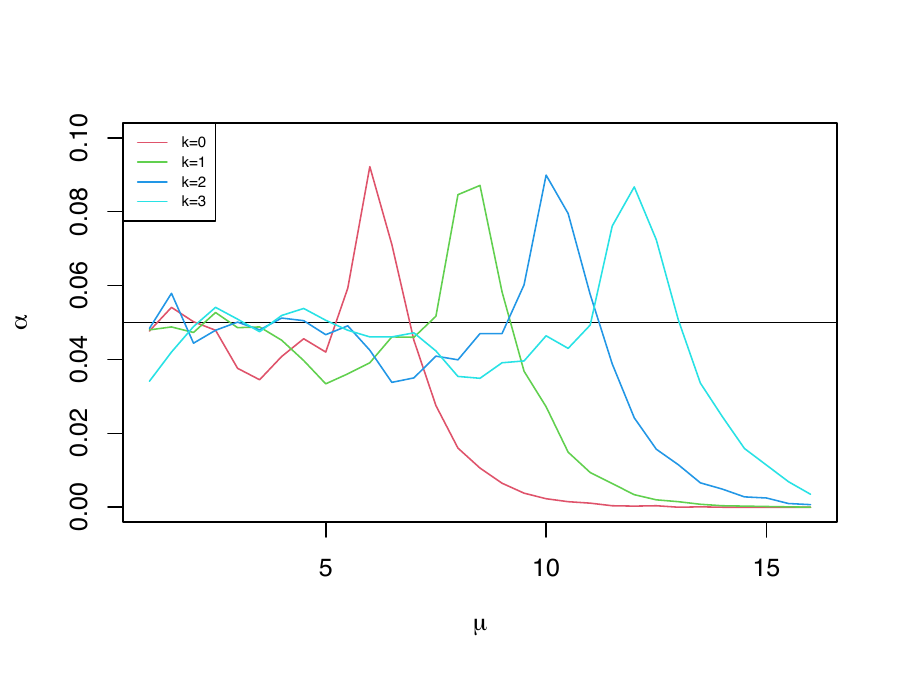}
	\caption{ {Proportion of rejections of the null hypothesis for the test based on $Z_{n,k}$} for $k=0,1,2,3$ and $n=50$ ($\alpha=0.05$).}
	\label{fig:figure1}
\end{figure}

Initially, given $\alpha=0.05$, the three tests are compared by means of their empirical significance level computed generating $10000$ samples of size $n=20, 50$ from Poisson distributions with $\mu$ varying from $1$ to $16$ by ${0.5}$. 
From Figure \ref{fig:figure2}, it is worth noting that, even for the moderate sample size $n = 20$, the test based on $W_n$ captures the nominal significance level satisfactory, highlighting a rather good speed of convergence to the normal distribution, also confirmed by the empirical level for $n=50$. The Fisher test shows an empirical significance level very close to the nominal one even for $n = 20$, except when $\mu$ is small.
The test based on $MN_n$, on the contrary, maintains the nominal level of significance rather closely only for $n=50$.

\begin{figure} 
	\centering
	\includegraphics[width=\textwidth]{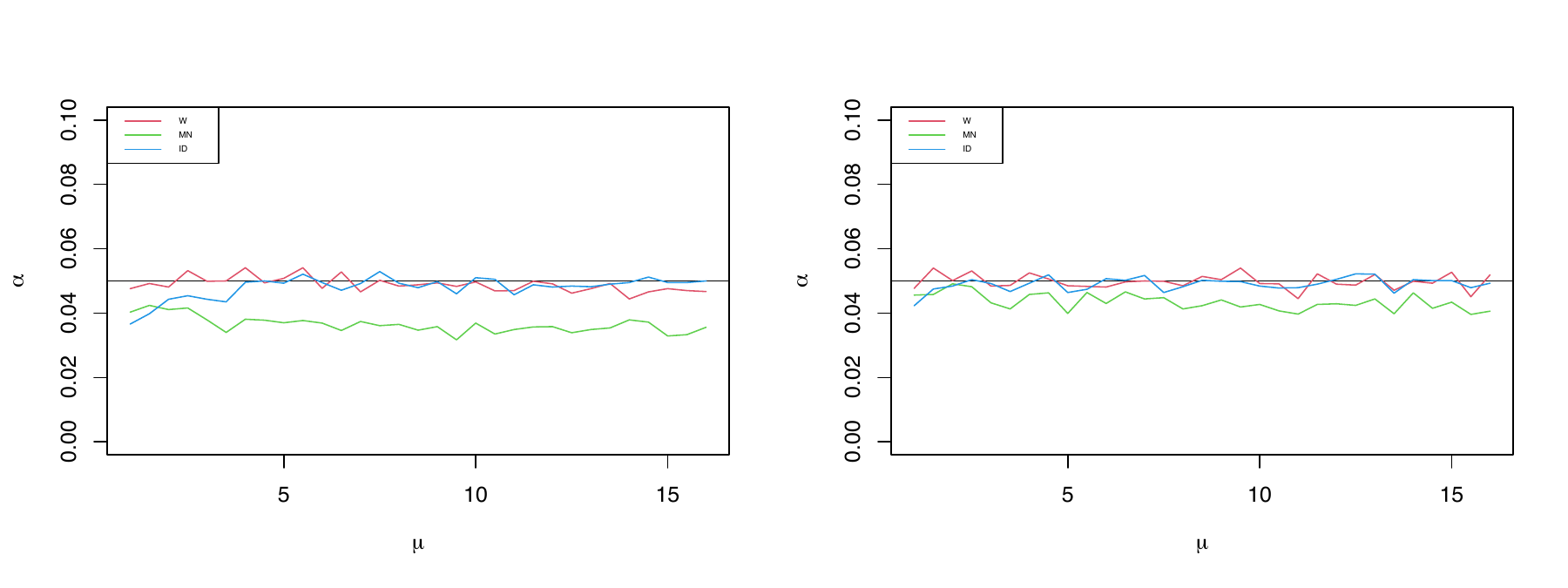}
	\caption{ {Proportion of rejections of the null hypothesis for the test based on $W_n$, $MN_n$ and $ID_n$} for $n=20$ (on the left) and $n=50$ (on the right).}
	\label{fig:figure2}
\end{figure}

The null hypothesis of Poissonity is tested against the following alternative models (for details see \citealp{Johnson}): {mixture of two Poisson} denoted by $\mathcal{MP}(\mu_{1},\mu_{2})$, {Binomial} by $\mathcal{B}(k, p)$, {Negative Binomial} by $\mathcal{NB}(k, p)$,	{Generalized Hermite} by $\mathcal{GH}(a, b, k)$, {Discrete Uniform} in $\{0,1,\dots,\nu\}$ by $\mathcal{DU}(\nu)$,  {Discrete Weibull} by  $\mathcal{DW}(q, \beta)$, {Logarithmic Series} translated by -1 by $\mathcal{LS^-}(\theta)$, {Logarithmic Series} by $\mathcal{LS}(\theta)$, {Generalized Poisson} denoted by $\mathcal{GP}(\mu_1,\mu_2)$, {Zero-inflated Binomial} denoted by $\mathcal{ZB}(k, p_{1}, p_{2})$, {Zero-inflated Negative Binomial} by $\mathcal{ZNB}(k, p_{1}, p_{2})$, {Zero-inflated Poisson} by $\mathcal{ZP}(\mu_{1},\mu_{2})$. Various parameters values are considered (see Table 1). Moreover, the significance level of the tests is reported for Poisson distributions with $\mu=0.5, 1, 2, 5, 10, 15$. The alternatives considered in the simulation study include overdispersed and underdispersed, heavy tails, mixtures and zero-inflated distributions together with distributions having mean close to variance. Some alternatives that do not belong to the class $\Delta$, such as the logarithmic and shifted-logarithmic with parameters 0.7, 0.8, and 0.9 {and the discrete uniform in $\{0,1,2,3\}$}, have been included to check the robustness of the $W_n$ and $MN_n$ tests.

From each distribution, $10000$ samples of size $n=20, 30, 50$ are independently generated and, on each sample, the three tests are performed. The empirical power of each test is computed as the percentage of rejections of the null hypothesis. The simulation is implemented by using \cite{Rcore} and in particular the packages \texttt{extraDistr}, \texttt{hermite} and  \texttt{RNGforGPD}. 

Simulation results are reported in Table \ref{tab: tabella1}. The $MN_n$ test is somewhat too conservative for smaller sample sizes and the $ID_n$ test does not capture the significance level for small $\mu$, while $W_n$ shows an empirical significance level rather close to the nominal one even for small sample size and small $\mu$.

\begin{table} 
	\centering
	\caption{Empirical power with $5\%$ nominal significance level.}
	\medskip
	\begin{scriptsize}
		\setlength{\tabcolsep}{1.8pt}
		\renewcommand\arraystretch{1.3}
		\begin{tabular}{lccccccccc} 
			\hline
			Model & $W_{20}$ & $MN_{20}$& $ID_{20}$ & $W_{30}$ & $MN_{30}$ & $ID_{30}$ &$W_{50}$ & $MN_{50}$& $ID_{50}$\\
			\hline
			$\Pi_{0.5}$  &    4.4     &       2.9     &       2.8     &       3.9     &       3.2     &       3.5     &       4.7     &       4.0     &       3.9 \\
			$\Pi_{1}$ &    5.5     &       4.5     &       4.2     &       4.6     &       4.2     &       4.2     &       4.6     &       4.8     &       4.5 \\
			$\Pi_{2}$ &    5.3     &       4.8     &       4.8     &       5.5     &       4.7     &       4.8     &       4.8     &       4.8     &       5.0 \\
			$\Pi_{5}$ &    5.4     &       3.7     &       5.2     &       5.6     &       4.3     &       4.8     &       5.2     &       4.6     &       5.0 \\
			$\Pi_{10}$ &    4.5     &       3.7     &       5.1     &       5.0     &       4.0     &       5.0     &       4.6     &       4.4     &       5.0 \\
			$\Pi_{15}$ &    4.3     &       3.6     &       5.1     &       4.7     &       4.0     &       4.8     &       4.9     &       4.3     &       5.1 \\
			$\mathcal{B}(1,0.5)$ &   58.6     &      41.2     &      25.9     &      81.3     &      70.3     &      70.3     &      99.2     &      98.3     &      96.6 \\
			$\mathcal{B}(4,0.25)$ &   11.6     &       9.1     &       6.8     &      14.6     &      13.2     &      12.3     &      21.9     &      22.8     &      21.3 \\
			$\mathcal{B}(30,0.1)$  & 5.9     &  3.7      &    4.8     & 5.8     &    4.5    &     5.4  &   6.4       &   5.7     &      6.7 \\
			$\mathcal{NB}(1,0.5)$ &   41.2     &      50.1     &      50.1     &      55.2     &      65.6     &      65.8     &      76.8     &      84.4     &      83.9 \\
			$\mathcal{NB}(4,0.75)$ &   11.2     &      15.2     &      16.6     &      14.1     &      20.5     &      22.2     &      21.1     &      29.6     &      32.2 \\
			$\mathcal{NB}(10,0.9)$ &    5.9     &       6.3     &       6.6     &       5.6     &       7.1     &       7.6     &       6.5     &       8.4     &       9.1 \\
			$\mathcal{GH}(1,1.25,2)$ &   18.6     &      50.2     &      43.5     &      27.0     &      65.2     &      58.4     &      43.7     &      84.1     &      78.8 \\
			$\mathcal{GH}(1,1.5,2)$ &   21.6     &      53.1     &      46.8     &      32.1     &      68.4     &      61.5     &      51.4     &      86.4     &      81.4 \\
			$\mathcal{GH}(1,1.75,2)$ &   22.3     &      54.9     &      48.5     &      35.1     &      70.1     &      64.0     &      58.9     &      87.8     &      83.8 \\
			$\mathcal{DU}(3)$ &    5.9     &       1.9     &       2.7     &       4.6     &       2.2     &       3.7     &       6.2     &       2.2     &       6.3 \\
			$\mathcal{DU}(5)$ &    6.8     &      18.7     &       4.9     &       9.4     &      26.8     &       6.2     &      15.6     &      41.9     &       8.9 \\
			$\mathcal{DU}(10)$     & 35.5    &  87.9     &    71.7    & 54.9    &    96.1   &    86.9  &   83.2      &   99.7    &      97.3 \\
			$\mathcal{DU}(15)$     & 56.6    &  98.1     &    95.3    & 78.9    &    99.8   &    99.1  &   97.4      &   100.0   &      100.0 \\
			$\mathcal{DW}(0.5,3)$ &   56.0     &      38.9     &      24.3     &      77.3     &      65.4     &      65.4     &      97.6     &      95.8     &      94.3 \\
			$\mathcal{DW}(0.8,5)$ &  100.0     &      99.7     &      99.0     &     100.0     &     100.0     &     100.0     &     100.0     &     100.0     &     100.0 \\
			$\mathcal{LS^-}(0.6)$  &   45.0     &      50.4     &      51.1     &      59.9     &      66.2     &      66.7     &      79.1     &      83.8     &      83.7 \\
			$\mathcal{LS^-}(0.7)$ &   63.3     &      70.0     &      70.1     &      79.0     &      84.9     &      85.0     &      93.9     &      96.3     &      96.0 \\
			$\mathcal{LS^-}(0.8)$ &   81.8     &      88.3     &      88.3     &      93.8     &      96.6     &      96.5     &      99.4     &      99.8     &      99.7 \\
			$\mathcal{LS^-}(0.9)$ &   94.9     &      98.7     &      98.6     &      99.2     &      99.9     &      99.9     &     100.0     &     100.0     &     100.0 \\
			$\mathcal{LS}(0.6)$  &   92.0     &      48.2     &      37.0     &      98.4     &      58.4     &      39.9     &     100.0     &      71.8     &      42.9 \\
			$\mathcal{LS}(0.7)$ &   79.8     &      25.4     &      32.1     &      92.0     &      27.1     &      36.0     &      99.1     &      29.8     &      42.5 \\
			$\mathcal{LS}(0.8)$ &   76.5     &      35.3     &      56.0     &      88.9     &      41.1     &      68.7     &      97.9     &      51.8     &      82.9 \\
			$\mathcal{LS}(0.9)$ &   91.7     &      83.2     &      91.5     &      97.8     &      92.9     &      97.4     &      99.9     &      98.9     &      99.8 \\
			$\mathcal{GP}(1,0.1)$  &    8.6     &      11.1     &      11.6     &       9.4     &      13.4     &      14.6     &      13.6     &      18.9     &      20.7 \\
			$\mathcal{GP}(3,0.25)$  &   21.9     &      42.8     &      45.6     &      31.7     &      56.7     &      59.6     &      48.2     &      75.8     &      79.3 \\
			$\mathcal{GP}(5,0.4)$  &   51.8     &      83.8     &      85.5     &      69.6     &      94.2     &      95.3     &      90.6     &      99.5     &      99.6 \\
			$\mathcal{ZB}(5,0.9,0.2)$ &   21.5     &      49.6     &      14.6     &      43.5     &      57.7     &      13.6     &      75.4     &      72.4     &      13.2 \\
			$\mathcal{ZB}(5,0.5,0.3)$ &   44.5     &      27.3     &       9.5     &      65.1     &      39.8     &      12.2     &      90.8     &      60.7     &      19.6 \\
			$\mathcal{ZB}(5,0.4,0.5)$ &   67.5     &      53.1     &      36.3     &      85.7     &      74.4     &      53.9     &      98.2     &      94.1     &      78.1 \\
			$\mathcal{ZNB}(5,0.9,0.1)$ &    6.9     &       7.3     &       7.7     &       8.3     &       9.7     &      10.1     &      10.6     &      12.8     &      13.4 \\
			$\mathcal{ZNB}(5,0.5,0.3)$ &   68.8     &      99.6     &      98.5     &      87.9     &     100.0     &      99.9     &      98.5     &     100.0     &     100.0 \\
			$\mathcal{ZNB}(10,0.4,0.5)$ &   97.5     &     100.0     &     100.0     &      99.9     &     100.0     &     100.0     &     100.0     &     100.0     &     100.0 \\
			$\mathcal{ZP}(1,0.2)$ &    8.5     &       8.6     &       8.1     &      10.9     &      11.9     &      11.5     &      17.1     &      16.9     &      14.0 \\
			$\mathcal{ZP}(1.5,0.3)$ &   28.2     &      26.6     &      22.0     &      39.6     &      38.6     &      31.8     &      62.5     &      59.4     &      48.5 \\
			$\mathcal{ZP}(2,0.5)$ &   77.8     &      74.4     &      65.1     &      91.9     &      90.4     &      83.6     &      99.4     &      98.9     &      96.8 \\
			\hline
		\end{tabular}  
	\end{scriptsize} 
	\label{tab: tabella1}
\end{table}

As expected, also from the theoretical results by \cite{janssen2000global}, none of the three tests shows performance superior to the others for any alternative and for any sample size, and their power crucially depends on the set of parameters also for alternatives in the same class. 
Obviously, when the alternative model is very similar to a Poisson r.v., e.g. when the alternative is Binomial with $k$ large and $p$ small, or when dealing with the Poisson Mixtures or the Negative Binomial with $k$ large, the power of all the tests predictably decreases. Low power is also observed against slightly overdispersed or underdispersed discrete uniform distributions, while the power rapidly increases as overdispersion becomes more marked, with the performance of all three tests becoming comparable as $n$ increases. 
For the Weibull distributions, the $W_n$ test has a certain edge over its competitors, which, on the other hand, perform better when the generalized Poisson distributions are considered, even though their power is satisfactory only for ${\mathcal GP}(5,0.4)$. 
The power of the test based on $W_n$ is the highest for all the logarithmic distributions, with less remarkable differences for $\theta=0.9$, while the three tests exhibit nearly the same power for the shifted log-normal distribution, where a decrease in the power of $W_n$ occurs especially for $n=20$.
As to the zero-inflated distributions, the three tests have a really unsatisfactory behaviour for $\mathcal{ZNB}(5, 0.9, 0.1)$ and $\mathcal{ZP}(1,0.2)$ also for $n=50$, but $W_n$ shows the best performance for most of the remaining alternatives and sample sizes. 
Overall, the number of alternatives for which the three tests reach a power greater than 90\% is 
almost the same for $n=20$ and $n=30$. Interestingly, for $n=50$ the proposed test reaches a power greater than 90\% more frequently not only than the straightforward Fisher test but also than the Meintanis test, which is more complex to be implemented. 

Finally, the discriminatory capability of the tests under contiguous alternatives is evaluated. In particular, the tests based on $W_n$, $MN_n$ and $ID_n$ are considered and, for sake of brevity, let $P_n$ be the power function corresponding to each test statistic. Obviously $P_n$ is a function of $\lambda$, where $\lambda\in \,\, ] 0,\sqrt{n}[$, which ensures that the contiguous mixture never completely degenerates, keeping its mixture nature for any $\lambda$.  Hence a basic efficiency measure is the following $$r_n=\frac{1}{\sqrt{n}}\int_{0}^{\sqrt{n}}P_n(\lambda)d\lambda,$$
 evidently $r_n \in ]0,1[$, and since $P_n$ is not known, the Monte Carlo estimate
$${\widehat r}_n=\frac{1}{m}\sum_{i=1}^{m}\widehat{P}_n(\lambda_i)$$ is considered, where {$\lambda_{i}=i\varepsilon$, with $i=1,\ldots,m$ and $m\leq\lfloor \frac{\sqrt{n}}{\varepsilon}\rfloor-1$,} for $\varepsilon$ sufficiently small, and $\widehat{P}_n$ is the empirical power.

To assess the performance of the three tests fairly, the alternative distributions of type \eqref{mix} are obtained by selecting $Y$ such that the tests achieve similar power when $Y$ is the alternative distribution. In particular, $Y$ is $\mathcal{B}(1, 0.5)$ and $X$ is $\Pi_{0.5}$.
In this case, it should be noted that the behaviour of $W_n$ coincides with that of the simpler version $Z_{n,0}$ since $k^*_n=0$ almost surely.
In Figure \ref{fig:figure3}, the empirical power as a function of $\lambda$, computed on 10000 independently generated samples, is reported for both $n=20$ and $n=50$ sample sizes and for $\varepsilon={0.25}$ and in Table \ref{tab: tabella2} the corresponding values of ${\widehat r}_n$ are reported.

\begin{figure}
	\centering
	\includegraphics[width=\textwidth]{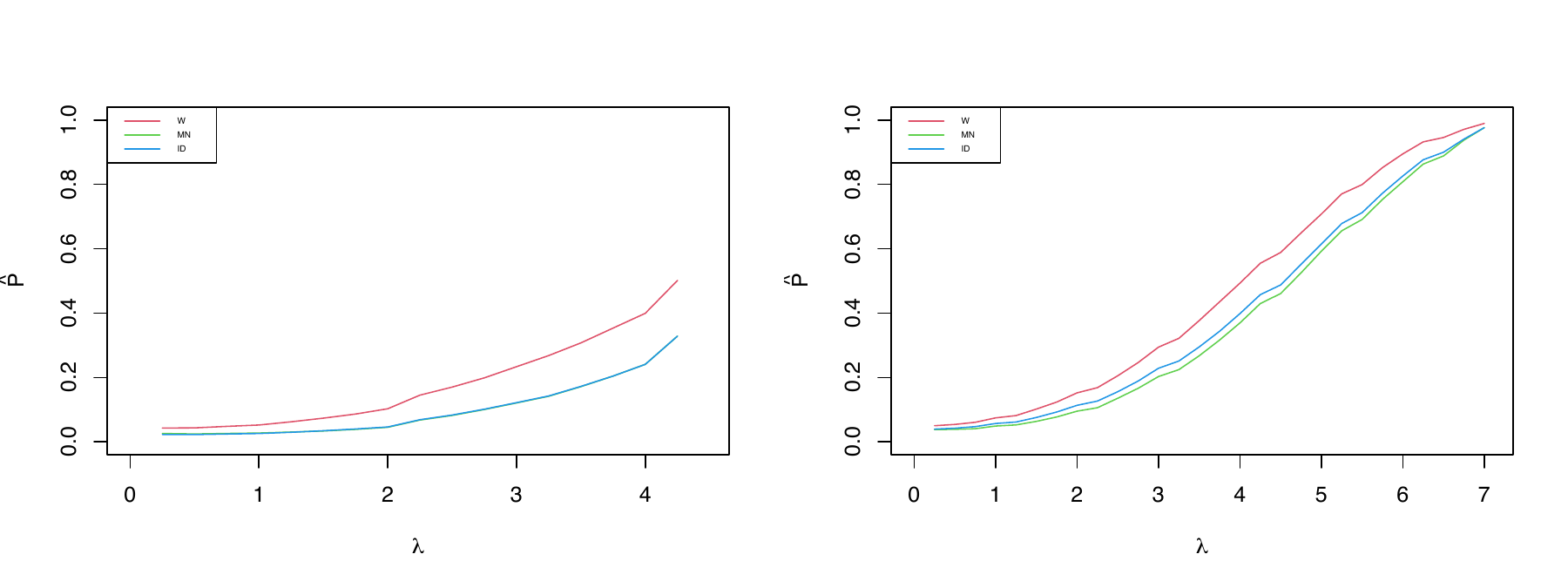}
	\caption{Empirical power {of $W_n$, $MN_n$ and $ID_n$} against shrinking alternative for $n=20$ (on the left) and $n=50$ (on the right).}
	\label{fig:figure3}
\end{figure} 

\begin{table} 
	\centering
	\caption{${\widehat r}_n$ with $n=20$ and $n=50$.}
	\medskip
      	\begin{tabular}{cccc}
			\hline
			& \multicolumn{1}{c}{$W_n$}  & $MN_n$ & $ID_n$\\
			\hline
			$n=20$	&{0.182} & {0.101}  & {0.101}\\
			$n=50$	&{0.461} &  {0.387}  &{0.404}\\
			\hline
		\end{tabular}
	\label{tab: tabella2}
\end{table}
Graphical and numerical results show that, even if all the tests improve as $n$ increases, the proposed test performs better for both sample sizes. In contrast, the $ID_n$ and $MN_n$ tests have very similar behaviour.

{\section{Some applications in biodosimetry}
	Biodosimetry, the measurement of biological response to radiation, plays an important role in accurately reconstructing the dose of radiation received by an individual by using biological markers, such as chromosomal abnormalities caused by radiation. When radiation exposure occurs, the damage in DNA is randomly
	distributed between cells producing chromosome aberrations and the interest is the number of aberrations (generally dicentrics and/or rings) observed.
	The Poisson distribution is the most widely recognised and commonly used distribution for the number of recorded
	dicentrics or rings per cell  \citep{ainsbury2013comparison} even though, due to the complexity of radiation exposure cases, other distributions may be suitably applied. Indeed, in presence of
	partial body irradiation, heterogeneous exposures, and exposure to high Linear Energy Transfer radiations, the Poisson distribution does not fit properly and the distribution of the chromosome aberrations provides useful insight about the patient’s exposure. Therefore, when dealing with data coming from the framework of biodosimetry, a first necessary step consists of testing Poissonity. }

{Following \cite{PuigWei}, we test Poissonity on the following datasets:
	\begin{itemize}
		\item[-] Dataset 1: number of chromosome aberrations (dicentrics and rings) from a patient, exposed to radiation after the nuclear accident of Stamboliyski (Bulgaria) in 2011;
		\smallskip
		\item[-] Dataset 2: total number of dicentrics from a male exposed to high doses of radiation caused by the nuclear accident happened in Tokai-mura (Japan) in 1999;
		\smallskip
		\item[-] Dataset 3: total number of rings from a male exposed to high doses of radiation caused by the nuclear accident happened in Tokai-mura (Japan) in 1999;  
		\smallskip
		\item[-] Dataset 4: number of dicentrics observed from a healthy donor when exposed to 5 Gy of X rays;
		\smallskip
		\item[-] Dataset 5: number of dicentrics observed from a healthy donor when exposed to 7 Gy of X rays.
	\end{itemize}
	Data are reported in Table \ref{tab: data} and the values of the test statistic, together with the corresponding p-values, are given in Table \ref{tab: risultati}.
	The test suggests that there are not noticeable departures from the Poisson distribution for Dataset 1 and Dataset 2, while for Dataset 3 the result of the test is statistically significant at $5\%$ level. Finally, the p-values of the test for Dataset 4 and Dataset 5 reveal a strong evidence against the null hypothesis of Poisson distributed data.}

\begin{table} [h!]
	    \centering
        \caption{Frequency of the number of aberrations for Datasets 1-5.}
		\medskip   
		\begin{tabular}{cccccc} 
			\hline
{\# aberrations}
			& Dataset 1  & Dataset 2 & Dataset 3 & Dataset 4 & Dataset 5 \\
			\hline
			0& 117 & 19 & 107 & 3& 0\\
			
			1& 94 & 17 & 42 & 23&  4\\
			
			2& 51 & 50 & 23 & 58&  23\\
			
			3& 15 & 40 & 3& 38&  35\\
			
			4& 6 & 23& $-$ & 15&  35\\
			
			5& 0 & 16& $-$ & 10&  29\\
			
			6& 0 & 4& $-$ & 2&  10\\
			
			7& 1 & 4& $-$ & 1&  9\\
			
			8& $-$ & 0& $-$ & $-$&  4\\
			
			9& $-$ & 2& $-$ & $-$&  1\\
			\hline
		\end{tabular}    

	\label{tab: data}
\end{table}

\begin{table} [h!]
	\centering
    \caption{Values of $W_n$ and p-values (in parenthesis) for the datasets in Table  \ref{tab: data}.}
	\medskip
    \begin{tabular}{lcccccccc} 
			\hline
			Dataset 1 & & Dataset 2 & & Dataset 3 & & Dataset 4 & & Dataset 5 \\
			\hline			
			-1.5391& &1.5377& &-1.9705& &4.0627& &3.127\\
			(0.1237)& &(0.1241) & & (0.0488)& &(0.0000)& &(0.0018)\\
			\hline
		\end{tabular} 

	\label{tab: risultati}
\end{table}

\section{Discussion}
Notwithstanding many tests for Poissonity are in literature, 
{the proposed family of test statistics seems to be an appealing alternative in the absence of prior information regarding the type of deviation from Poissonity. In particular, the  statistics are rather simple and easily interpretable and the test implementation does not require intensive computational effort. Moreover,} the test is consistent against any fixed alternative when $k$ is equal to $0$ and when it is selected using the data-driven criterion, that is $k=k^*_n$. For $k=0$ the test statistic basically compares an estimator of $P(X = 0)$ assuming that $X$ is Poisson with the relative frequency of $0$ but the finite sample performance of the test may not be satisfactory, especially for small sample size and relatively large Poisson parameter. The performance improves for $k^*_n$, when the test juxtaposes the plug-in estimator of the cumulative distribution function of a Poisson r.v. and the empirical cumulative distribution function in $k^*_n$. Indeed, even if $k^*_n$ converges a.s. to $0$, the convergence rate may be very slow for large values of the Poisson parameter, and thus, even for large sample sizes, $k^*_n$ can be rather larger than $0$.
{Finally, the simulation study shows that, with respect to the test by \cite{Meintanis} and that based on the Fisher index of dispersion, the test based on $k^*_n$ offers a rather satisfactory protection against a range of alternatives.}

\bibliographystyle{apalike}      
\bibliography{Riferimentibib}

\end{document}